\documentclass{amsart}
\usepackage{amssymb,amscd,amsmath,amsthm}

\usepackage{tikz}
\usetikzlibrary{decorations.pathreplacing}

\usepackage{blkarray}
\usepackage{array}

\theoremstyle{plain}
\newtheorem{definition}{Definition}[section]
\newtheorem{theorem}{Theorem}[section]
\newtheorem{proposition}[theorem]{Proposition}
\newtheorem{lemma}[theorem]{Lemma}
\newtheorem{corollary}[theorem]{Corollary}

\theoremstyle{remark}
\newtheorem{remark}[theorem]{Remark}
\newtheorem{example}[theorem]{Example}

\newcommand{\FF}{\mathbb{F}_2}
\newcommand{\gn}[1]{\langle {#1}\rangle}
\newcommand{\set}[1]{\left\{#1\right\}}
\newcommand{\abs}[1]{\left\vert#1\right\vert}
\newcommand{\EH}{\widetilde{\mathcal{H}}}

\usepackage[foot]{amsaddr}

\title{The weighted poset metrics and directed graph metrics}

\author{Jong Yoon Hyun\textsuperscript{1}}
\email{hyun33@kias.re.kr}
\address{\textsuperscript{1}
Korea Institute for Advanced Study (KIAS),
85, Hoegiro, Dongdaemun-Gu, Seoul, 02455,
Republic of Korea.}
\thanks{
The first named
author is supported by the National Research Foundation of
Korea(NRF) grant funded by the Korea government (MEST)
(2014R1A1A2A10054745)}

\author{Hyun Kwang Kim\textsuperscript{2}}
\email{hkkim@postech.ac.kr}
\address{\textsuperscript{2} 
Department of Mathematics, 
Pohang University of Science and Technology (POSTECH),
77 Cheongam-Ro, Nam-Gu, Pohang, Gyeongbuk, 37673,
Republic of Korea.}

\author{Jeong Rye Park\textsuperscript{3}}
\email{parkjr@pusan.ac.kr}
\address{\textsuperscript{3} 
Department of Mathematics, 
College of Natural Sciences, 
Pusan National University, 
2, Busandaehak-ro 63beon-gil, Geumjeong-gu, Busan, 46241, 
Republic of Korea.}

\begin{document}

\maketitle

\begin{abstract}

Recently, in \cite{EFM}, the authors introduced metrics on $\FF^n$ based on directed graphs on $n$ vertices and developed some basic coding theory on directed graph metric spaces. 
In this paper, we consider the problem of classifying directed graphs which admit the extended Hamming codes to be a perfect code.
We first consider weighted poset metrics as a natural generalization of poset metrics and investigate interrelation between weighted poset metrics and directed graph based metrics. 
In the next, we classify weighted posets on a set with eight elements
and directed graphs on eight vertices which admit the extended Hamming code $\EH_3$ to be a $2$-perfect code. 
We also construct some families of such structures
for any $k \geq 3$ which can be viewed as a generalizations of some results in~\cite{EFM} and~\cite{HK04}. 
Those families enable us to construct packing or covering codes of radius 2 under certain maps.
	
\end{abstract}

\section{Introduction} 
\label{sec:introduction}

Let $\mathbb{F}_q$ be the finite field with $q$ elements 
and $\mathbb{F}_q^n$ be the vector space of $n$-tuples 
over $\mathbb{F}_q$. 
Coding theory may be considered as the study of $\mathbb{F}_q^n$ 
when $\mathbb{F}_q^n$ is endowed with the Hamming metric. 
Since the late 1980's several attempts have been
made to generalize the classical problems of the coding theory by
introducing a new non-Hamming metric on $\mathbb{F}_q^n$ 
(cf.~\cite{N1}-\cite{S}).
These attempts led Brualdi et al.~\cite{BGL} 
to introduce the concept of poset metrics $\mathbb{F}_q$ in $1995$. 
Recently, in \cite{EFM}, the authors introduced metrics on $\FF^n$
based on directed graphs on $n$ vertices.
The authors discussed basic topics of coding theory such as, 
isometry groups, reconstruction problems,
and the MacWillams Identity and the MacWillams Extension Properties 
in directed graph metric spaces. 
In this paper, we consider the problem of classifying directed graphs 
which admit the extended Hamming codes to be a perfect code.
We first consider weighted poset metrics as a natural generalization of poset metrics and investigate interrelation between weighted poset metrics and directed graph based metrics. 
In the next, we classify weighted posets a set with eight elements
and directed graphs on eight vertices which admit the extended Hamming code $\EH_3$ to be a $2$-perfect code. 
We also construct some families of such structures
for any $k \geq 3$ which can be viewed as a generalizations of some results in~\cite{EFM} and~\cite{HK04}. 
Those families enable us to construct packing or covering codes of radius 2 under certain maps.

\medskip

Let $\FF$ be the finite field of order two and $\FF^n$
the vector space of binary $n$-tuples.
The support of $x$ in $\FF^n$ is
the set of non-zero coordinate positions.
The Hamming weight $w_H(x)$ of a vector $x$ in $\FF^n$ is
the number of non-zero coordinate positions.
Throughout this paper, we identify $x$ in $\FF^n$ with its support.
\medskip

Let $d_*$ be a metric based on an object $*$ which is one of
a poset, a weighted poset and a digraph.
Let $x$ be a vector in $\FF^n$ and $r$ a non-negative integer.
The $*$-sphere with center $x$ and radius $r$ is defined as the set
\[S_*(x;r)=\set{y\in\FF^n\mid d_*(x,y)\leq r}\]
of all vectors in $\FF^n$
whose $*$-distance from $x$ is at most $r$.
\medskip

Let $(\FF^n,d_*)$ denote the metric space on $\FF^n$ endowed with the $d_*$-metric.
A subset $C$ of $(\FF^n,d_*)$ is called a $*$-code of length $n$.

\begin{definition}
Let $C$ be a $*$-code of length $n$.\\
We say that $C$ is an $r$-covering $*$-code
if the union of $*$-spheres of radius $r$
centered at the codewords of $C$ is $\FF^n$.
The covering radius of $C$ is the smallest $r$
such that $C$ is an $r$-covering code.\\
We say that $C$ is an $r$-packing $*$-code
if $*$-spheres of radius $r$ centered at the codewords of $C$
are pairwise disjoint.
The packing radius of $C$ is the largest $r$
such that $C$ is an $r$-packing code.\\
We say that $C$ is an $r$-error-correcting perfect  (for short, $r$-perfect)
$*$-code if it is an $r$-covering and $r$-packing $*$-code.
\end{definition}

We start with by introducing two types of non-Hamming metrics, namely weighted poset metrics and directed graph metrics, which will be discussed in this paper.


\subsection{Poset Metrics} 
\label{sub:poset_metrics}

Let $(P,\preceq)$ be a partially ordered set
(for short, poset) of size $n$.
A subset $I$ of $P$ is called an order ideal
if $i\in I$ and $j\preceq i$ imply that $j\in I$.
For a subset $A$ of $P$, $\gn{A}_P$ denotes the smallest order ideal of $P$ containing $A$.
The order ideal generated by $\set{i}$ is denoted $\gn{i}_P$ for short.
\medskip

Without loss of generality,
we may assume that $P$ is $\set{1,2,\ldots,n}$
and that the coordinate positions of vectors in $\FF^n$
are labeled by $P$.
The {\it $P$-weight} of a vector $x$ in $\FF^n$ is defined by the size
of the smallest order ideal of $P$ containing $x$, that is,
\[w_p(x)=\abs{\gn{x}_P}.\]
The {\it $P$-distance} of the elements $x$ and $y$ in $\FF^n$ is defined by
\[d_P(x,y)=w_P(x-y).\]
The metric $d_P$ on $\FF^n$,  which is
introduced by Brualdi et al. in~\cite{BGL}, is called a {\it poset metric}.


\subsection{Weighted Poset Metrics} 
\label{sub:weighted_poset_metrics}

Let $(P,\preceq)$ be a partially ordered set and $\pi$ a function  from $P$ to $\mathbb{N}$,
where $\mathbb{N}$ is the set of natural numbers. The triple $(P,\preceq, \pi)$ is called a $\pi$-weighted poset.
We simply denote a weighted poset by $P_\pi$.
The {\it $P_\pi$-weight} of $x$ in $\FF^n$ is defined by
\[w_{P_\pi}(x)=\sum_{i\in\gn{x}_{P}}\pi(i).\]
The {\it $P_\pi$-distance} of the vectors $x$ and $y$ in $\FF^n$ is defined by
\[d_{P_\pi}(x,y)=w_{P_\pi}(x-y).\]

Notice that \[d_{P_\pi}(x,y)=d_{P}(x,y)\] when the weight function $ \pi: P \rightarrow \mathbb{N}$ is given by $\pi(i) = 1$ for any $i$ in $P$.

\begin{lemma}
	If $P_\pi$ is a weighted poset,
	then $P_\pi$-distance $d_{P_\pi}$ is a metric on $\FF^n$.
\end{lemma}
\begin{proof}
Obviously, $P_\pi$-distance is symmetric and positive definite.
To prove that $d_{P_\pi}(x,y)\leq d_{P_\pi}(x,z)+d_{P_\pi}(z,y)$
for all $x,y$ and $z$, it is sufficient to show that $P_\pi$-weight satisfies
the triangle inequality
$w_{P_\pi}(x+y)\leq w_{P_\pi}(x)+w_{P_\pi}(y)$.
Since $\gn{x+y}_P\subseteq\gn{x}_P\cup\gn{y}_P$, we obtain
\begin{eqnarray*}
w_{P_\pi}(x+y) 	&\leq& \sum_{i\in\gn{x+y}_{P}}\pi(i)
			 	\leq \sum_{i\in\gn{x}_{P}}\pi(i)
			 				+\sum_{i\in\gn{y}_{P}}\pi(i) \\
			 	&=& w_{P_\pi}(x)+w_{P_\pi}(y)
\end{eqnarray*}
\end{proof}

We call the metric $d_{P_\pi}$
on $\FF^n$ a {\it$\pi$-weighted poset metric}.
In ~\cite{BS}, the authors introduced weighted Hamming metrics and constructed an infinite family of weighted Hamming metrics
which make the generalized Goppa codes to be perfect codes. The weighted Hamming metrics form an important class of weighted poset metrics in which
the posets are given by antichains.

\medskip

In the weighted poset metric space, we can compute the cardinality of spheres of radius $r$ as follows.
Let $\Omega_j^\omega(i)$ be the number of order ideals of $P_{\pi}$
of size $i$,  $P_\pi$-weight $\omega$ with $j$ maximal elements. Recall that $P_\pi$-weight of a subset is defined to be the sum of
weights of its elements. Then the number of vectors in $\FF^n$ whose $P_\pi$-distance to the zero vector is exactly $\omega$ equals

\[\left\{
\begin{array}{ll}
	1				& \text{if}~\omega=0,\\
	\displaystyle\sum_{i=1}^\omega\sum_{j=1}^i2^{i-j}\Omega_j^\omega(i)& \text{if}~\omega>0.
\end{array} \right.\]
Therefore we have
\begin{equation}\label{P_G:sphere_size_radius_r}
	\abs{S_{P_\pi}(x;r)}=1+\sum_{\omega=1}^r\sum_{i=1}^\omega\sum_{j=1}^i2^{i-j}\Omega_j^\omega(i).
\end{equation}


\subsection{Graph Metrics} 
\label{sub:graph_metrics}
All graphs considered in this paper are directed simple graphs,
that is, neither loops nor multiple edges are allowed.
We refer to~\cite{WDB} for general facts on directed graphs.
Let $G$ be a digraph consisting of a vertex set $V(G)$ and an edge set $E(G)$
where the (directed) edge is an ordered pair of distinct vertices.
For any $u$ and $v$ in $V(G)$ with $u\neq v$,
we say that $v$ is {\it dominated by $u$}
if there is a path from $u$ to $v$.
For a subset $S$ of $V(G)$,
we denote $\gn{S}_G$ the set consisting of $S$ and the vertices
which are dominated by vertices in $S$.
The {\it $G$-weight} of $x$ in $\FF^n$ is defined by
the size of the vertices dominated by $\set{x}$, that is,
\[w_G(x)=\abs{\gn{x}_G}.\]
The {\it $G$-distance} of the vectors $x$ and $y$ in $\FF^n$
is defined by
\[d_G(x,y)=w_G(x-y).\]
The metric $d_G$ on $\FF^n$  which is
introduced by Etzion et al. in~\cite{EFM} is called a {\it directed graph metric}. The Hamming metric maybe considered as a $G$-metric
where $G$ contains no edges.



\section{Weighted posets and Digraphs} 
\label{sec:weighted_posets_and_digraphs}
In this section, we discuss the interrelation between weighted poset metrics and directed graph metrics.
We introduce the weight poset metric induced by a directed graph and the directed graph metric induced by a weighted poset.
We conclude that these two concepts are `almost' the same.


\subsection{Induced weighted poset metrics from digraphs} 
\label{sub:induced_weighted_poset_metrics_from_digraphs}
In this subsection,
we define the weighted poset induced by a digraph.\medskip

Let $G$ be a digraph with $n$ vertices.
For any $u$ and $v$ in $V(G)$,
we write $u\sim v$ if $u=v$ or
there are paths from $u$ to $v$ and from $v$ to $u$.
One can easily check that it is an equivalence relation on $V(G)$.

Let $G/\mathbin\sim$ be the set of equivalence classes of the relation $\sim$.
Let $\bar{v}$ denote the equivalence class containing a vertex $v$. We define a relation on $G/\mathbin\sim$ by
$\bar{u}\preceq\bar{v}$ if $\bar{u}=\bar{v}$ or
there is a path from $a$ to $b$ where $a\in\bar{v}$ and $b\in\bar{u}$.

\begin{lemma}\label{G-poset}
	Let $G$ be a digraph.
	Then $\preceq$ is well-defined 
	and it is a partial order on $G/\mathbin\sim$.
\end{lemma}

\begin{proof}
The proof is straightforward.
\end{proof}

We may make the poset $(G/\mathbin\sim,\preceq)$ into a weighted poset by using a weight function
$\tilde{\pi}: G/\mathbin\sim \rightarrow \mathbb{N}$ by $\tilde{\pi}(\bar{v})=\abs{\bar{v}}$. Here, for a set $A$, $|A|$
denotes the size of $A$. The weighted poset $(G/\mathbin\sim,\preceq, \tilde{\pi})$ is called the weighted poset induced by $G$. In the sequel,
we simply denote it by $P_{\tilde{\pi}}$ instead of $(G/\mathbin\sim,\preceq, \tilde{\pi})$.
Notice that $P_{\tilde{\pi}}$ is a poset with size $m$ where $m$ denotes the number of
equivalence classes of $G/\mathbin\sim$.
Note also that the $P_{\tilde{\pi}}$-weight of $x$ in $\FF^m$,
is given by
\[w_{P_{\tilde{\pi}}}(x)=\sum_{\bar{v}\in\gn{x}_{P_{\tilde{\pi}}}}\abs{\bar{v}},\]
and the $P_{\tilde{\pi}}$-distance of the vectors $x$ and $y$ in $\FF^m$
is written by
\[d_{P_{\tilde{\pi}}}(x,y)=w_{P_{\tilde{\pi}}}(x-y).\]
\medskip

We mention that a directed acyclic graph corresponds to a poset metric.

\begin{example}
In Fig.\ref{ex1}, the digraph with four vertices has a directed cycle of length $3$.
Then the induced weighted poset $P_{\tilde{\pi}}$ has two elements.
\begin{figure}[h]
	\begin{minipage}{1in}
		\begin{tikzpicture}[auto,node distance=1cm,semithick]
			\tikzstyle{vertex}=[circle,fill,inner sep=1.5pt]
			
			\node[vertex] (a)  			   {};
			\node[vertex] (b) [below of=a] {};
			\node[vertex] (c) [below of=b,xshift=-8mm] {};
			\node[vertex] (d) [below of=b,xshift=8mm]  {};

			\path[->]
		    (a) edge (b)
			(b) edge (c)
			(c) edge (d)
			(d) edge (b);

			\draw
			(a) node[right, scale=0.8] {$~1$}
			(b) node[right, scale=0.8] {$~2$}
			(c) node[left, scale=0.8]  {$3~$}
			(d) node[right, scale=0.8] {$~4$};
		\end{tikzpicture}
	\end{minipage}
\hskip1.5cm
	\begin{minipage}{1in}%
		\begin{tikzpicture}[auto,node distance=1cm,semithick]
			
			\node (a) 				{};
			\node (b) [right of=a]  {};

			\path[-stealth]
			(a) edge[double] node[above]{$P_{\tilde{\pi}}$} (b);

		\end{tikzpicture}
	\end{minipage}
	\begin{minipage}{1in}%
		\begin{tikzpicture}[auto,node distance=1cm,semithick]
			\tikzstyle{vertex}=[circle,fill,inner sep=1.5pt]
			
			\node[vertex] (a)  			   {};
			\node[vertex] (b) [below of=a] {};

			\path[-]
		    (a) edge (b);

		    \draw (a) node[right, scale=0.8]
		    				{$~\bar{1},\tilde{\pi}(\bar{1})=1$};
			\draw (b) node[right, scale=0.8]
							{$~\bar{2},\tilde{\pi}(\bar{2})=3$};

		\end{tikzpicture}
	\end{minipage}
	\caption{}
	\label{ex1}
\end{figure}

In Fig.\ref{ex2},
the induced weighted poset has four elements which are
labeled by the $P_{\tilde{\pi}}$-weight.

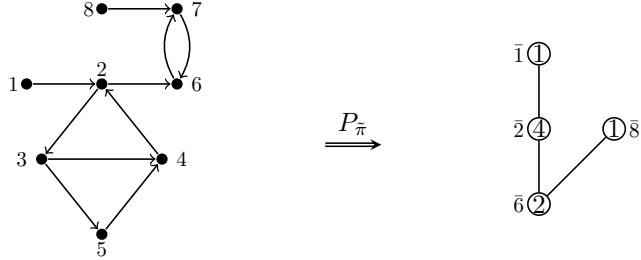
\begin{figure}[h]
	\begin{minipage}{1in}
		\begin{tikzpicture}[auto,node distance=1cm,semithick]
			\tikzstyle{vertex}=[circle,fill,inner sep=1.5pt]
			
			\node[vertex] (1)  			   {};
			\node[vertex] (2) [right of=1] {};
			\node[vertex] (3) [below of=2,xshift=-8mm] {};
			\node[vertex] (4) [below of=2,xshift=8mm]  {};
			\node[vertex] (5) [below of=3,xshift=8mm]  {};
			\node[vertex] (6) [right of=2] {};
			\node[vertex] (7) [above of=6] {};
			\node[vertex] (8) [above of=2] {};
		
			\path[->]
		    (1) edge (2)
			(2) edge (3)
			(3) edge (4)
			(3) edge (5)
			(5) edge (4)
			(4) edge (2)
			(2) edge (6)
			(8) edge (7);

			\path[->]
			(6) edge [bend left] (7)
			(7) edge [bend left] (6);

			\draw
			(1) node[left,  scale=0.8] {$~1$}
			(2) node[above, scale=0.8] {$2$}
			(3) node[left,  scale=0.8] {$3~$}
			(4) node[right, scale=0.8] {$~4$}
			(5) node[below, scale=0.8] {$5$}
			(6) node[right, scale=0.8] {$~6$}
			(7) node[right, scale=0.8] {$~7$}
			(8) node[left,  scale=0.8] {$~8$};
		\end{tikzpicture}
	\end{minipage}
\hskip1.5cm
	\begin{minipage}{1in}%
		\begin{tikzpicture}[auto,node distance=1cm,semithick]
			
			\node (a) 				{};
			\node (b) [right of=a]  {};

			\path[-stealth]
			(a) edge[double] node[above]{$P_{\tilde{\pi}}$} (b);

		\end{tikzpicture}
	\end{minipage}
	\begin{minipage}{1in}%
		\begin{tikzpicture}[auto,node distance=1cm,semithick]
			\tikzstyle{weight}=[shape=circle,draw,inner sep=0.1pt]
			
			\node[weight] (1)  			   {1};
			\node[weight] (2) [below of=1] {4};
			\node[weight] (6) [below of=2] {2};
			\node[weight] (8) [right of=2] {1};
		
			\path[-]
		    (1) edge (2)
		    (2) edge (6)
		    (8) edge (6);

		    \draw (1) node[left, scale=0.8] {$\bar{1}~$};
			\draw (2) node[left, scale=0.8] {$\bar{2}~$};
			\draw (6) node[left, scale=0.8] {$\bar{6}~$};
			\draw (8) node[right,scale=0.8] {$~\bar{8}$};
		\end{tikzpicture}
	\end{minipage}
	\caption{$\tilde{\pi}(\bar{1})=\tilde{\pi}(\bar{8})=1$,
				$\tilde{\pi}(\bar{2})=4$ and $\tilde{\pi}(\bar{6})=2$.}
	\label{ex2}
\end{figure}
\end{example}


\subsection{Induced graph metrics from weighted posets} 
\label{sub:induced_graph_metrics_from_weighted_posets}
In this subsection,
we define the digraph induced by a weighted poset.\medskip

Let $P_\pi$ be a $\pi$-weighted poset.
We denote $E(P_\pi)$ the set of pairs $(a,b)$ of elements of $P_\pi$
such that $b\preceq a$.
For $a$ in $P_\pi$,
we set $a_0=a$ and $\hat{a}=\set{a_0,a_1,\ldots,a_{\pi(a)-1}}$.
Denote by $E(\hat{a})$ the set of pairs $(a_i,a_{i+1})$,
where the subscripts are taken modulo $\pi(a)$.
Note that the digraph consisting of the vertex set $\hat{a}$
and the edge set $E(\hat{a})$ is a directed cycle graph.
We define a digraph $G_\pi$ by
\[V(G_\pi)=\bigcup_{a\in P_\pi}\hat{a},~
E(G_\pi)=E(P_\pi)\cup[\bigcup_{a\in P_\pi} E(\hat{a})].\]

\begin{remark} 
	For a given digraph $G$, it is not generally equal to $G_{\tilde{\pi}}$
	(cf.~Fig.\ref{rectangle}).
	\begin{figure}[h]
	\hskip1.5cm
	\begin{minipage}{1in}%
		\begin{tikzpicture}[auto,node distance=1cm,semithick]
			\tikzstyle{vertex}=[circle,fill,inner sep=1.5pt]
	
			\node[vertex] (a)  			   {};
			\node[vertex] (b) [below of=a] {};
			\node[vertex] (c) [left of=b] {};
			\node[vertex] (d) [above of=c] {};

			\path[->]
		    (b) edge (a)
			(b) edge node [below,yshift=-3mm] {$G$} (c)
			(c) edge (d)
			(a) edge (d);

			\path[->]
			(d) edge (b);

		\end{tikzpicture}
	\end{minipage}
\hskip1cm
	\begin{minipage}{1in}
		\begin{tikzpicture}[auto,node distance=1cm,semithick]
			\tikzstyle{weighted}=[shape=circle,draw,inner sep=0.1pt]
			
			\node			(a) {};
			\node[weighted] (b) [below of=a,yshift=5mm]{$4$};
			\node		    (c) [below of=b,yshift=-1mm] {$P_{\tilde{\pi}}$};
		\end{tikzpicture}
	\end{minipage}
\hskip.5cm
	\begin{minipage}{1in}%
		\begin{tikzpicture}[auto,node distance=1cm,semithick]
			\tikzstyle{vertex}=[circle,fill,inner sep=1.5pt]
		
			\node[vertex] (a)  			   {};
			\node[vertex] (b) [right of=a] {};
			\node[vertex] (c) [below of=b] {};
			\node[vertex] (d) [below of=a] {};

			\path[->]
		    (a) edge (b)
			(b) edge (c)
			(c) edge node [below,yshift=-3mm] {$G_{\tilde{\pi}}$} (d)
			(d) edge (a);
		
		\end{tikzpicture}
	\end{minipage}
	\caption{The weighted poset $P_{\tilde{\pi}}$ is induced by the digraph $G$. There is only one element such that $P_{\tilde{\pi}}$-weight is $4$.}
	\label{rectangle}
	\end{figure}
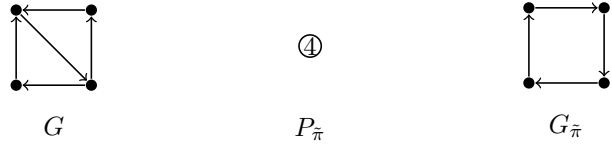
\end{remark}

The diagram of metrics introduced in this paper is given in Fig.\ref{all_matric}.\\
Recall that $G_{\pi}$ is
the digraph induced by a $\pi$-weighted poset
and that $P_{\tilde{\pi}}$ is
the $\tilde{\pi}$-weighted poset induced by a digraph.

\begin{figure}[h]
	\begin{tikzpicture}[auto,node distance=1.5cm,semithick]
    \tikzstyle{style}=[shape=rectangle,rounded corners,draw,
    text centered,align=center]

	\node		 (c)                         {};

    \node[style] (w) [left of=c,xshift=-25mm]
    					{$\pi$-weighted poset metrics\\$P_\pi$-code};
    \node[style] (p) [below of=c,yshift=-1cm]
    					{poset metrics\\$P$-code};
    \node[style] (h) [below of=p,yshift=-1cm]
    					{Hamming metrics\\$H$-code};
    \node[style] (g) [right of=c,xshift=25mm]
    					{graph metrics\\$G$-code};

    \path[thick,->]
   		(w) edge node [left,xshift=-3mm] {$\pi=id$}  (p)
   		(p) edge node [left] {antichain} (h);

    \path[thick,->]
    	(g) edge node [right,xshift=1mm,yshift=-2mm]
    					{no cycles\\(acyclic)} (p);

    \path[thick,->]
      	(g) edge [bend left]
      			 node [below,xshift=5.5mm] {no edges}(h);

    \path[thick,->]
    (g) edge node [above		    ] {induce $P_{\tilde\pi}$}(w)
    (w) edge [bend left] node [above] {induce $G_\pi$} (g);
    \end{tikzpicture}
    \caption{A diagram of metrics}
    \label{all_matric}
\end{figure}
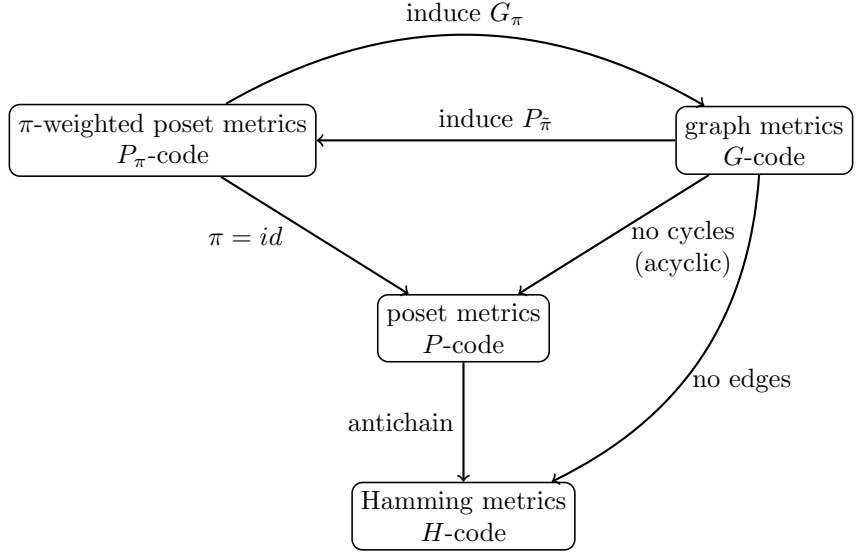


\subsection{The relationship between weighted posets and digraphs} 
\label{sub:the_relationship_between_weighted_posets_and_digraphs}

In this subsection, we deal with covering codes and packing codes
for a graph metric and a weighted poset metric.
To do this, we define two mappings to obtain new codes
from their images.

\medskip

Let $G$ be a digraph with $n$ vertices and
$P_{\tilde{\pi}}$ the weighted poset induced by $G$.
We denote the ground set of $P_{\tilde{\pi}}$ by
$\set{\bar{1},\bar{2},\ldots,\bar{m}}$.

Let $x=(X_{\bar{1}},X_{\bar{2}},\ldots,X_{\bar{m}})$
be an element of $\FF^n$,
where $X_{\bar{i}}$ is a binary vector of length $\abs{\bar{i}}$.
The mapping $\varphi_{\tilde{\pi}}$ from $\FF^n$ to $\FF^m$ is defined by
$x^{\varphi_{\tilde{\pi}}}=(x_{\bar{1}},x_{\bar{2}},\ldots,x_{\bar{m}})$ in $\FF^m$,
where 
\begin{center}
$x_{\bar{i}}=\left\{
\begin{array}{ll}
0  & \text{if}~X_{\bar{i}}=0,\\
1  & \text{otherwise.}
\end{array} \right.$
\end{center}
Let $C$ be a $G$-code. 
Then $C^{\varphi_{\tilde{\pi}}} = \{x^{\varphi_{\tilde{\pi}}}: x \in C \}$ is a $P_{\tilde{\pi}}$-code.
We point out that 
the linearity does not preserved under the map $\varphi_{\tilde{\pi}}$.

\medskip

Let $P_\pi$ be a weighted poset with $m$ elements with total weight $n$ and
$G_\pi$ a digraph induced by $P_\pi$.
We denote the ground set of $P_\pi$ by
$\set{\hat{1},\hat{2},\ldots,\hat{m}}$.

Let $x=(x_{\hat{1}},x_{\hat{2}},\ldots,x_{\hat{m}})$
be an element of $\FF^m$.
The mapping $\varphi_\pi$ from $\FF^m$ to $\FF^n$ is defined by
$x^{\varphi_\pi}=(X_{\hat{1}},X_{\hat{2}},\ldots,X_{\hat{m}})$
in $\FF^n$,
where $X_{\hat{i}}=(x_{\hat{i}},0,0,\ldots,0)$ is a vector
of length $\pi(\hat{i})$.
Notice that if $C$ is a ${P_\pi}$-code of length $m$,
then $C^{\varphi_\pi}$ is a $G_{\pi}$-code of length $n$
and the linearity of $C$ is preserved by $\varphi_{\pi}$.

\begin{lemma}\label{pro:sig-tau}
	For $x,y\in\FF^n$ and $u,v\in\FF^m$ the following statements hold:
	\begin{enumerate}
		\item $w_G(x)=w_{P_{\tilde{\pi}}}(x^{\varphi_{\tilde{\pi}}})$.
		\item $w_{P_{\tilde{\pi}}}(x^{\varphi_{\tilde{\pi}}}+y^{\varphi_{\tilde{\pi}}})
		\leq
		w_{P_{\tilde{\pi}}}((x+y)^{\varphi_{\tilde{\pi}}})$.
		\item $w_{P_\pi}(u)=w_{G_\pi}(u^{\varphi_\pi})$.
		\item $u^{\varphi_\pi}+v^{\varphi_\pi}=(u+v)^{\varphi_\pi}$.
	\end{enumerate}
\end{lemma}
\begin{proof}
$\mathrm{(i)}$ We claim that
\[\gn{x}_G=\bigcup_{\bar{i}\in\gn{x^{\varphi_{\tilde{\pi}}}}_{P_{\tilde{\pi}}}}\bar{i}.\]
Let $a$ be in $\gn{x}_{G}$.
Then there is a path from $v$ to $a$ such that $x_v=1$, and so
$\bar{v}$ in $\gn{x^{\varphi_{\tilde{\pi}}}}_{P_{\tilde{\pi}}}$,
where $\bar{v}$ is the equivalence class containing $v$.
The claim holds since the equivalence class containing $a$
is also in $\gn{x^{\varphi_{\tilde{\pi}}}}$.
The converse can be proved by a similar argument.

\noindent $\mathrm{(ii)}$ 
Let $x=(X_{\bar{1}},X_{\bar{2}},\ldots,X_{\bar{m}})$
and $y=(Y_{\bar{1}},Y_{\bar{2}},\ldots,Y_{\bar{m}})$.
Assume that $\bar{i}$ contributes zero to $w_{P_{\tilde{\pi}}}((x+y)^{\varphi_{\tilde{\pi}}})$, i.e.,
$ X_{\bar{i}} + Y_{\bar{i}} =0$. Then we have $X_{\bar{i}} = Y_{\bar{i}}$ and hence $x_{\bar{i}} = y_{\bar{i}}$.
Therefore $\bar{i}$ contributes zero to $w_{P_{\tilde{\pi}}}(x^{\varphi_{\tilde{\pi}}}+y^{\varphi_{\tilde{\pi}}}).$
One can easily check that if $\bar{i}$ contributes a positive weight to $w_{P_{\tilde{\pi}}}((x+y)^{\varphi_{\tilde{\pi}}})$
then $\bar{i}$ contributes either zero or a positive weight to $w_{P_{\tilde{\pi}}}(x^{\varphi_{\tilde{\pi}}}+y^{\varphi_{\tilde{\pi}}}).$
Hence we may have a strict inequality in $\mathrm{(ii)}$.

\noindent$\mathrm{(iii)}$ 
We claim that
\[\bigcup_{\hat{i}\in\gn{u}_{P_\pi}}\hat{i}=\gn{u^{\varphi_\pi}}_{G_\pi}.\]
Let $a$ be an element of the left hand side in the claim.
By the definition of an ideal of a weighted poset,
there exists $\hat{s}$ in $P_\pi$ such that $u_{\hat{s}}=1$
and $\hat{a}\preceq\hat{s}$,
where $\hat{a}$ is the element in $P_\pi$ containing $a$.
Then $\hat{s}\subseteq\gn{u^{\varphi_\pi}}_{G_\pi}$ and hence $a\in\gn{u^{\varphi_\pi}}_{G_\pi}$.
The converse can be proved by a similar argument.

\noindent$\mathrm{(iv)}$
Let $u^{\varphi_\pi}=(U_{\hat{1}},U_{\hat{2}},\ldots,U_{\hat{m}})$
and $v^{\varphi_\pi}=(V_{\hat{1}},V_{\hat{2}},\ldots,V_{\hat{m}})$.
Then we have\\ the $\hat{i}$-th coordinate of $u+v$ is $0$ if and only if $u_{\hat{i}} = v_{\hat{i}}$\\
if and only if $U_{\hat{i}} = V_{\hat{i}}$ if and only if the $\hat{i}$-th coordinate of $u^{\varphi_\pi}+u^{\varphi_\pi}$ is $0$.
\end{proof}

We now establish the relationship between a covering $G$-code (resp., a packing $P_\pi$-code $C$)
and a covering $P_{\tilde{\pi}}$-code (resp., a packing $G_\pi$-code).

\begin{theorem}\label{covering_packing}
	The following statements hold:
	\begin{enumerate}
		\item If $C$ is an $r$-covering $G$-code,
				then $C^{\varphi_{\tilde{\pi}}}$ is an $r$-covering $P_{\tilde{\pi}}$-code.
		\item If $C$ is an $r$-packing ${P_\pi}$-code,
				then $C^{\varphi_\pi}$ is an $r$-packing $G_\pi$-code.
	\end{enumerate}
\end{theorem}
\begin{proof}
$\mathrm{(i)}$
Let $x\in\FF^m$. Then $x^{\varphi_{\pi}} \in \FF^n$.
Since $C$ is an $r$-covering $G$-code,
there is a codeword $c$ in $C$
such that $x^{\varphi_\pi}\in S_G(c;r)$.
So $w_G(c+x^{\varphi_\pi})\leq r$.
Applying Lemma~\ref{pro:sig-tau} $(i)$, $(ii)$ to $w_G(c+x^{\varphi_\pi})$, we have
\[w_{P_{\tilde{\pi}}}(c^{\varphi_{\tilde{\pi}}}+(x^{\varphi_\pi})^{\varphi_{\tilde{\pi}}})\leq w_{P_{\tilde{\pi}}}((c+x^{\varphi_\pi})^{\varphi_{\tilde{\pi}}})=w_G(c+x^{\varphi_\pi})\leq r.\]
Thus $x\in S_{P_{\tilde{\pi}}}(c^{\varphi_{\tilde{\pi}}};r)$ because $x=(x^{\varphi_\pi})^{\varphi_{\tilde{\pi}}}$.

\noindent$\mathrm{(ii)}$
Let $x\in S_{G_\pi}(c_1^{\varphi_\pi};r)\cap S_{G_\pi}(c_2^{\varphi_\pi};r)$ for some distinct $c_1^{\varphi_\pi}$ and $c_2^{\varphi_\pi}$ in $C^{\varphi_\pi}$,
where $c_1$ and $c_2$ belong to $C$. Notice that $c_1\neq c_2$
because of the definition of $\varphi_\pi$.
Then $w_{G_\pi}(c_1^{\varphi_\pi};x)\leq r$ and $w_{G_\pi}(c_2^{\varphi_\pi};x)\leq r$.
Applying Lemma~\ref{pro:sig-tau} $(i)$, $(ii)$ to $w_{G_\pi}(c_i^{\varphi_\pi};x)$, we have
\[w_{P_\pi}((c_i^{\varphi_\pi})^{\varphi_{\tilde{\pi}}}+x^{\varphi_{\tilde{\pi}}})\leq w_{P_\pi}((c_i^{\varphi_\pi}+x)^{\varphi_{\tilde{\pi}}})=w_{G_\pi}(c_i^{\varphi_\pi}+x)\leq r,\]
where $i=1,2$.
Due to $c_i=(c_i^{\varphi_\pi})^{\varphi_{\tilde{\pi}}}$, we have $x^{\varphi_{\tilde{\pi}}}\in S_{P_\pi}(c_1;r)\cap S_{P_\pi}(c_2;r)$,
which is a contradiction to the assumption that $C$ is an $r$-packing ${P_\pi}$-code.
\end{proof}



\section{Weighted posets and digraphs which admits the extended Hamming code to be a perfect code} 
\label{sec:extended_hamming_codes}
In this section,
we consider the problem of construction of wighted posets and digraphs
which admit the extended Hamming code $\EH_k$ to be a $2$-perfect code. We classify all such structures when $k=3$,
and construct some families of such structures for any $k \geq 3$.

\medskip

We start with a quick introduction of extended Hamming codes.

Let $\EH_k=[n=2^k,2^k-1-k,4]$ $(k\geq2)$ be the extended Hamming code
with
the usual parity check matrix $H_k$.
The $H_k$ is a $(k+1)\times2^k$ binary matrix
whose first row is the all one vector of length $2^k$.
The remaining $k$ rows of $H_k$ form a $k\times2^k$ submatrix
whose $i$-th column corresponds to
the $2$-adic representation of $i-1$.
For example, $\EH_3$ is the $[8,4,4]_H$ code
with the parity check matrix $H_3$, where $H_3$ is given by

$
\begin{blockarray}{cccccccc}
	1 & 2 & 3 & 4 & 5 & 6 & 7 & 8\\
	\begin{block}{(cccccccc)}
	 	1 & 1 & 1 & 1 & 1 & 1 & 1 & 1 \\
	 	0 & 0 & 0 & 0 & 1 & 0 & 1 & 1 \\
	 	0 & 0 & 1 & 1 & 0 & 0 & 1 & 1 \\
	 	0 & 1 & 0 & 1 & 0 & 1 & 0 & 1 \\
	\end{block}
\end{blockarray}
$~.

The following proposition, which gives a necessary and sufficient condition for a given code to be a perfect code, was proved in~\cite{HK04} for poset codes.
We first generalize it to weighted codes and digraph codes.

\begin{proposition}\label{main:iff}
	Let $*$ be either a $\pi$-weighted poset $P_{\pi}$ or a digraph $G$, and
	$C$ be an $[n,k]$ binary linear $*$-code.
	Then $C$ is an $r$-perfect $*$-code
	if and only if the following two conditions are satisfied:
	\begin{enumerate}
		\item (The sphere packing condition)
				$\abs{S_*(0;r)}=2^{n-k}$,
		\item (The partition condition)
				for any non-zero codeword $c$ and any partition
				$\set{x,y}$ of $c$,
				either $w_*(x)\geq r+1$
				or $w_*(y)\geq r+1$.
	\end{enumerate}
\end{proposition}
\begin{proof}
(Sufficiency)~$\mathrm{(i)}$
By $*$-perfectness, we have $\abs{C}\cdot\abs{S_*(c;r)}=\abs{\FF^n}$
so that $\abs{S_*(c;r)}=2^{n-k}$.

\noindent$\mathrm{(ii)}$
Assume that $\set{x,y}$ is a partition of $c$ such that
$w_*(x)\leq r$ and $w_*(y)\leq r$.
Then $x$ and $y$ are in $S_*(0;r)\cap S_*(c;r)$,
a contradiction to $*$-perfectness.\\
(Necessity) By~$\mathrm{(i)}$, it is sufficient to show that
$S_*(0;r)\cap S_*(c;r)=\phi$ for any non-zero $c\in C$.
Let $\alpha\in S_*(0;r)\cap S_*(c;r)$.
Then $w_*(\alpha)\leq r$ and $w_*(\alpha+c)\leq r$.
Since $\set{\alpha\cap c, c\setminus \alpha\cap c}$ is a partition
of $c$, the partition condition implies that either
$w_*(\alpha\cap c)\geq r+1$ or $w_*(c\setminus\alpha\cap c)\geq r+1$.
Since $w_*(\alpha\cap c)\leq w_*(\alpha)\leq r$,
we have $w_*(c\setminus \alpha\cap c)\geq r+1$.
Then
\begin{eqnarray*}
	r\geq w_*(\alpha+c) &=& w_*(\alpha\setminus\alpha\cap c
							+c\setminus\alpha\cap c) \\
				  &\geq& w_*(c\setminus \alpha\cap c)\geq r+1,
\end{eqnarray*}
a contradiction.
\end{proof}

\begin{corollary}\label{main:double}
Let $*$ be either a $\pi$-weighted poset $P_{\pi}$ or a digraph $G$.
	The extended binary Hamming code $\EH_k$ is a
	$2$-packing $*$-code if and only if
	for any codeword $c$ of $\EH_k$ with $w_*(c)=4$,
	and any partition $\set{x,y}$ of $c$ such that $w_H(x)=w_H(y)=2$,
	we have either $w_*(x)\geq3$ or $w_*(y)\geq3$.
\end{corollary}
\begin{proof}
It is straightforward from the partition condition of
Proposition~\ref{main:iff}.
\end{proof}

\begin{lemma}\label{main:weight2}
Let $*$ be either a $\pi$-weighted poset $P_{\pi}$ or a digraph $G$.
	If the extended Hamming code $\EH_k$
	is a $2$-perfect $*$-code,
	then there are no elements in $*$
	whose $*$-weight is bigger than two.
\end{lemma}
\begin{proof}
Assume that there is an element $x$ in $*$
such that $w_*(x)\geq3$.
By the $*$-perfectness, there is a codeword $c$ in $\EH_k$
such that $x\in S_*(c;2)$.
Then $c$ is a non-zero codeword and $w_H(c)\geq4$ since $\EH_k$ has the minimum Hamming distance $4$.
Thus we should have that $d_*(c,x)\geq3$, a contradiction.
\end{proof}

In the next two subsections, we classify weighted poset structures and digraphs
which admit the extended binary Hamming code $\EH_3$  to be
$2$-perfect. Motivated by these classifications we find some families of weighted poset structures and digraphs
which admit the extended binary Hamming code $\EH_k$ to be
$2$-perfect for any $k \geq 3$.

\begin{remark}
	Let $T$ be a linear isometry of $(\FF^n,d_*)$, where $*$ is either a $P_{\pi}$ or $G$. Then the image of an $r$-perfect $*$-code under $T$ is also $r$-perfect. It means that to find an $r$-perfect $*$-code we should consider in general, a labeling of coordinate positions.
\end{remark}

\subsection{Weighted posets} 
\label{sub:weighted_posets}

In this subsection, we consider a weighted poset $P_\pi$.
It follows from~(\ref{P_G:sphere_size_radius_r}) that
\[\abs{S_{P_\pi}(0;2)}
	=1+\Omega_1^1(1)+\Omega_1^2(1)+2\Omega_1^2(2)+\Omega_2^2(2).\]

\begin{lemma}\label{P_G:size}
	If the extended Hamming code $\EH_k$
	is a $2$-perfect $P_\pi$-code, then
	$\Omega_1^2(1)=1+\frac{1}{2}\Omega_1^1(1)(\Omega_1^1(1)-3)$.
\end{lemma}
\begin{proof}
Let $s$ be the number of elements in $P_\pi$
whose $P_\pi$-weight are $1$,{}
i.e. $s=\Omega_1^1(1)$.
It follows from the sphere packing condition and
(\ref{P_G:sphere_size_radius_r}) that
\begin{equation}\label{P_G:sphere_size}
	\abs{S_{P_\pi}(0;2)}=2^{k+1}
	=1+{s\choose1}+{s\choose2}+\Omega_1^2(1)+2\Omega_1^2(2).
\end{equation}
We have $n=2^k$ because $P_\pi$ is a weighted poset with $n$ elements. It follows that
\begin{equation}\label{P_G:poset_size}
	\abs{P_\pi}=2^k
	=s+\Omega_1^2(1)+\Omega_1^2(2).
\end{equation}
The proof is followed by solving (\ref{P_G:sphere_size}) and~(\ref{P_G:poset_size}).
\end{proof}

As an illustration of our theorem,
we classify weighted poset structures
which admit the extended binary Hamming code
$\EH_3$ to be a $2$-perfect $P_\pi$-code.

\begin{lemma}\label{P_G:1-lebel}
	If the extended Hamming code $\EH_3$
	is a $2$-perfect $P_\pi$-code, then
	$1\leq\Omega_1^1(1)\leq4$.
\end{lemma}
\begin{proof}
It follows from~(\ref{P_G:poset_size}) that
	\begin{equation}\label{P_G:poset_size_k=3}
			\abs{P_\pi}=2^3
				=\Omega_1^1(1)+\Omega_1^2(1)+\Omega_1^2(2).
	\end{equation}
Note that $\Omega_1^2(2)$ is a non-negative integer.
By Lemma~\ref{P_G:size}, the equality does not hold
if $\Omega_1^1(1)=0$ and $\Omega_1^1(1)\geq5$.
\end{proof}

It follows from Lemma~\ref{P_G:size}, Lemma~\ref{P_G:1-lebel}
and~(\ref{P_G:poset_size_k=3}) that
possible structure vectors are as follows:
\[(\Omega_1^1(1),\Omega_1^2(1),\Omega_1^2(2))
				=(1,0,7),(2,0,6),(3,1,4),(4,3,1).\]
The wighted poset $P_\pi$ is just a poset if $\Omega_1^1(1)\leq2$.
The poset structures
which admit the extended binary Hamming code
to be a $2$-perfect poset code were classified  in~\cite{HK04}.
Therefore we now classify weighted poset structures for the remaining cases $(3,1,4)$ and $(4,3,1)$.
\medskip

We define
\[\Delta_i^\omega:=\set{x\in P_\pi
					\mid\abs{\gn{x}_{P_\pi}}=i,w_{P_\pi}(x)=\omega},\]
i.e. $\Omega_1^\omega(i)=\abs{\Delta_i^\omega}$.

\noindent 
For any $a$ in $P_\pi$, we define
\[\Delta_i^\omega(a):=\set{x\in P_\pi
							\mid x\succeq a}\cap\Delta_i^\omega.\]

\noindent
Case $\mathrm{(i)}$ $\Omega_1^1(1)=4$:

\noindent
There is a unique weighted poset structure $P_\pi$
up to equivalence with the structure vector $(4,1,3)$~(cf.~Fig.\ref{4,3,1}).
\begin{figure}[h]
	\begin{tikzpicture}[auto,node distance=1cm,semithick]
	\tikzstyle{weight}=[shape=circle,draw,inner sep=0.1pt]

	\node[weight] (a) 			 {$1$};
	\node		  (b) [right of=a] {};
	\node		  (c) [right of=b] {};
	\node		  (d) [right of=c] {};
	\node[weight] (e) [right of=d] {$2$};	
	\node[weight] (f) [right of=e] {$2$};
	\node[weight] (g) [right of=f] {$2$};

	\node[weight] (h) [below of=a] {$1$};
	\node[weight] (i) [below of=b] {$1$};
	\node[weight] (j) [below of=c] {$1$};
	\node[weight] (k) [below of=d] {$1$};

	\path[-]
    (a) edge (h);
	
	\end{tikzpicture}
		\caption{There are three elements of which the $P_\pi$-weight is two and five elements of which the $P_\pi$-weight is one.}
        \label{4,3,1}%
\end{figure}
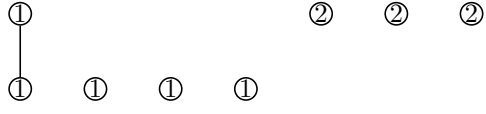

\noindent
Let us give a labeling as follows:
	$\Delta_1^1=\set{\alpha,\beta,\gamma,\delta'}$,
	$\Delta_1^2=\set{\beta',\gamma',\delta}$,
	and
	$\Delta_2^2(\alpha)=\set{\alpha'}$,
	where $\set{\alpha,\beta,\gamma,\delta}$,
			$\set{\alpha,\beta,\alpha',\beta'}$,
			$\set{\alpha,\gamma,\alpha',\gamma'}$ and
			$\set{\alpha,\delta,\alpha',\delta'}$ 
			are codewords of $\EH_3$.
Then we can check that the conditions of Corollary~\ref{main:double}
are satisfied.

\medskip

\noindent
Case $\mathrm{(ii)}$ $\Omega_1^1(1)=3$:

\noindent
There are four non-equivalent weighted poset structures,
say $P_\pi^1$, $P_\pi^2$, $P_\pi^3$ and $P_\pi^4$
with the structure vector $(3,1,4)$
(cf.~Fig.\ref{3,1,4} and Fig.\ref{P_G:cf}).

\begin{figure}[h]
	\begin{minipage}{1in}%
		\begin{tikzpicture}[auto,node distance=1cm,semithick]
			\tikzstyle{weight}=[shape=circle,draw,inner sep=0.1pt]

			\node[weight] (a)			     {$1$};
			\node[weight] (b) [right of=a] {$1$};
			\node[weight] (c) [right of=b] {$1$};
			\node[weight] (d) [right of=c] {$1$};
			
			\node[weight] (e) [right of=d] {$2$};	

			\node[weight] (f) [below of=b] {$1$};
			\node[weight] (g) [below of=c] {$1$};
			\node[weight] (h) [below of=d] {$1$};

			\path[-]
    		(a) edge (f)
    		(b) edge (f)
    		(c) edge (f)
    		(d) edge (f);

		\end{tikzpicture}
	\end{minipage}
\hskip3cm
	\begin{minipage}{2in}%
		\begin{tikzpicture}[auto,node distance=1cm,semithick]
			\tikzstyle{vertex}=[circle,fill,inner sep=1.5pt]
			\tikzstyle{weight}=[shape=circle,draw,inner sep=0.1pt]

			\node[weight] (a)	   		     {$1$};
			\node[weight] (b) [right of=a] {$1$};
			\node[weight] (c) [right of=b] {$1$};
			\node[weight] (d) [right of=c] {$1$};
			
			\node[weight] (e) [right of=d] {$2$};	
		
			\node[weight] (f) [below of=b] {$1$};
			\node[weight] (g) [below of=c] {$1$};
			\node[weight] (h) [below of=d] {$1$};
		
			\path[-]
		    (a) edge (f)
		    (b) edge (f)
		    (c) edge (g)
		    (d) edge (g);
		
		\end{tikzpicture}
	\end{minipage}
	\caption{}
	\label{3,1,4}
\end{figure}

\begin{figure}[h]
	\begin{minipage}{1in}%
		\begin{tikzpicture}[auto,node distance=1cm,semithick]
			\tikzstyle{weight}=[shape=circle,draw,inner sep=0.1pt]

			\node[weight] (a)			   {$1$};
			\node[weight] (b) [right of=a] {$1$};
			\node[weight] (c) [right of=b] {$1$};
			\node[weight] (d) [right of=c] {$1$};
			
			\node[weight] (e) [right of=d] {$2$};	

			\node[weight] (f) [below of=b] {$1$};
			\node[weight] (g) [below of=c] {$1$};
			\node[weight] (h) [below of=d] {$1$};

			\path[-]
    		(a) edge (f)
    		(b) edge (f)
    		(c) edge (f)
    		(d) edge (h);

		\end{tikzpicture}
	\end{minipage}
\hskip3cm
	\begin{minipage}{2in}%
		\begin{tikzpicture}[auto,node distance=1cm,semithick]
			\tikzstyle{weight}=[shape=circle,draw,inner sep=0.1pt]
		
			\node[weight] (a)		  	     {$1$};
			\node[weight] (b) [right of=a] {$1$};
			\node[weight] (c) [right of=b] {$1$};
			\node[weight] (d) [right of=c] {$1$};
			
			\node[weight] (e) [right of=d] {$2$};	
		
			\node[weight] (f) [below of=b] {$1$};
			\node[weight] (g) [below of=c] {$1$};
			\node[weight] (h) [below of=d] {$1$};
		
			\path[-]
		    (a) edge (f)
		    (b) edge (f)
		    (c) edge (g)
		    (d) edge (h);
		
		\end{tikzpicture}
	\end{minipage}
	\caption{}
	\label{P_G:cf}
\end{figure}

\noindent
Let $\set{\alpha,\beta,\gamma,\delta}$ be a codeword of $\EH_3$.
For each weighted poset, we give a labeling such that
$\Delta_1^1=\set{\alpha,\beta,\gamma}$.
Then by Corollary~\ref{main:double}, we should give a labeling such that
$\Delta_1^2=\set{\delta}$.
For each $P_\pi^1$ and $P_\pi^2$,
we give a labeling at the remaining coordinates as follows:
\begin{enumerate}
	\item[$P_\pi^1$:]
			$\Delta_2^2(\alpha)=\set{\alpha',\beta',\gamma',\delta'}$;
	\item[$P_\pi^2$:]
			$\Delta_2^2(\alpha)=\set{\alpha',\beta'}$,
			$\Delta_2^2(\beta)=\set{\gamma',\delta'}$,
\end{enumerate}
where $\set{\alpha,\beta,\alpha',\beta'}$,
			$\set{\alpha,\gamma,\alpha',\gamma'}$ and
			$\set{\alpha,\delta,\alpha',\delta'}$ are codewords.
Then $P_\pi^1$ and $P_\pi^2$
admit $\EH_3$ to be a $2$-perfect code by checking the partition condition in Proposition ~\ref{main:iff}.\\
For the remaining weighted posets,
we claim that they do not admit
$\EH_3$ to be a $2$-perfect code.
Without loss of generality,
we may assume that
$\abs{\Delta_2^2(\alpha)}\geq 2$ and $\abs{\Delta_2^2(\beta)}=1$,
and we denote the element in $\Delta_2^2(\beta)$ by $x$.
Let $\set{\alpha,\beta,x,y}$ be a codeword.
The codewords which contains $\set{\alpha, \beta}$
are $\set{\alpha,\beta,\alpha',\beta'}$ or
$\set{\alpha,\beta,\gamma',\delta'}$.
For the weighted poset $P_\pi^3$,
we can easily see that $y$ in $\Delta_2^2(\alpha)$,
which is a contradiction to Corollary~\ref{main:double}.
For the weighted poset $P_\pi^4$,
we have $y\in\Delta_2^2(\gamma)$.
Obviously, there exists $z$ in $\Delta_2^2(\alpha)$
such that $\set{\alpha,\gamma,y,z}$ is a codeword.
The codewords which contains $\set{\alpha, \gamma}$
are $\set{\alpha,\gamma,\alpha',\gamma'}$ or
$\set{\alpha,\gamma,\beta',\delta'}$,
which is a contradiction to Corollary~\ref{main:double}.\medskip

By combining the preceding discussion and the result ~\cite{HK04}, we obtain the following theorem.

\begin{theorem}
	Let $\EH_3$ denote the binary extended $[8,4,4]_H$ Hamming code.
	Then $\EH_3$ is a $2$-perfect $P_\pi$-code if and only if
	$P_\pi$ is one of the weighted posets described 
	in~Fig.\ref{4,3,1}, Fig.\ref{3,1,4} or Fig.\ref{HK_poset}.
\end{theorem}

\begin{figure}[h]
	\begin{minipage}{1in}%
		\begin{tikzpicture}[auto,node distance=1cm,semithick]
			\tikzstyle{weight}=[shape=circle,draw,inner sep=0.1pt]
		
			\node[weight] (1)							{$1$};
			\node		  (0) [above of=1,yshift=5mm]   {$\cdots$};

			\node[weight] (b) [right of=0]  {$1$};
			\node[weight] (c) [right of=b]  {$1$};
			
			\node[weight] (d) [left of=0]   {$1$};
			\node[weight] (e) [left of=d]   {$1$};

			\path[-]
			(1) edge (b)
			(1) edge (c)
			(1) edge (d)
			(1) edge (e);

			\draw[decorate,decoration={brace,mirror,raise=6pt,amplitude=10pt}, thick] 
			(c)--(e);
			\draw (0) node[above,yshift=6mm] {$7$};

		\end{tikzpicture}
	\end{minipage}
\hskip3cm
	\begin{minipage}{2in}%
		\begin{tikzpicture}[auto,node distance=1cm,semithick]
			\tikzstyle{weight}=[shape=circle,draw,inner sep=0.1pt]

			\node[weight] (21)                {$1$};
			\node[weight] (11) [right of=21,xshift=5mm]  {$1$};
			\node		  (01) [above of=11,yshift=5mm]   {$\cdots$};

			\node[weight] (b1) [right of=01]  {$1$};
			\node[weight] (c1) [right of=b1]  {$1$};
			
			\node[weight] (d1) [left of=01]   {$1$};
			\node[weight] (e1) [left of=d1]   {$1$};

			\path[-]
			(11) edge (b1)
			(11) edge (c1)
			(11) edge (d1)
			(11) edge (e1);

			\draw[decorate,decoration={brace,mirror,raise=6pt,amplitude=10pt}, thick] 
			(c1)--(e1);
			\draw (01) node[above,yshift=6mm] {$6$};
		
		\end{tikzpicture}
	\end{minipage}
	\caption{Ordinary posets which admit $\EH_3$ to be a 2-error correcting perfect code (cf.~\cite{HK04}).}
	\label{HK_poset}
\end{figure}
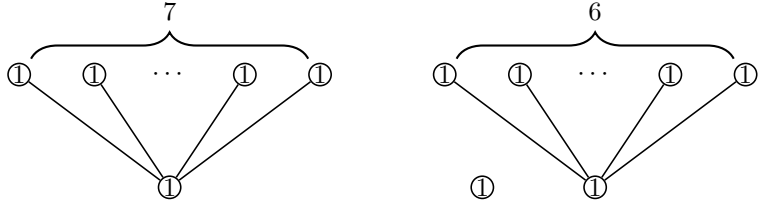

We provide two families of weighted poset structures
which admit the extended binary Hamming code $\EH_k$ to be
a $2$-perfect $P_\pi$-code.

By Lemma~\ref{P_G:size} and~(\ref{P_G:poset_size}),
if $\EH_k$ is a $2$-perfect $P_\pi$-code, then
we have the following possible structures vector:
\[(\Omega_1^1(1),\Omega_1^2(1),\Omega_1^2(2))
			=(1,0,2^k-1),(2,0,2^k-2),(3,1,2^k-4),(4,3,2^k-7),\ldots.\]

\begin{theorem}
	Let $P_\pi$ be a weighted poset satisfying one of the following:
	\begin{enumerate}
	\item $(\Omega_1^1(1),\Omega_1^2(1),\Omega_1^2(2))=(3,1,2^k-4)$,\\
			$\Delta_1^1=\set{\alpha,\beta,\gamma}$,
			$\Delta_1^2=\set{\delta}$,
			$\Delta_2^2(\alpha)=P_\pi^{(1)}
					\setminus(\Delta_1^1\cup\Delta_1^2)$;
	\item $(\Omega_1^1(1),\Omega_1^2(1),\Omega_1^2(2))=(4,3,2^k-7)$,\\
				$\Delta_1^1=\set{\alpha,\beta,\gamma,\delta'}$,
				$\Delta_1^2=\set{\beta',\gamma',\delta}$,
				$\Delta_2^2(\alpha)=P_\pi^{(1)}
					\setminus(\Delta_1^1\cup\Delta_1^2)$,
	\end{enumerate}
	where $\set{\alpha,\beta,\gamma,\delta}$,
				$\set{\beta,\beta',\gamma,\gamma'}$,
				$\set{\beta,\beta',\delta',\delta'}$ and
				$\set{\gamma,\gamma',\delta,\delta'}$ are codewords.
	Then the extended Hamming code $\EH_k$ is a
	$2$-perfect $P_\pi$-code.
\end{theorem}
\begin{proof}
We can check as in the preceding discussion that
Proposition~\ref{main:iff} and Corollary~\ref{main:double}
are satisfied.
\end{proof}

We provide concretely a possible labeling on a poset $P_\pi$
with Table~\ref{P_G:tab}
for which $\EH_3$ is a $2$-perfect
$P_\pi$-code.

\begin{table}[h]
	\small
	\caption{}
	\label{P_G:tab}
	\begin{tabular}{p{2cm}|p{10cm}}
		\hline
		Structure & Possible labeling \\
		\hline
		$(3,1,4)$ & $\Delta_1^1=\set{1,2,3}$,
					$\Delta_1^2=\set{4}$,
					$\Delta_2^2(1)=\set{5,6,7,8}$\\
		$(3,1,4)$ & $\Delta_1^1=\set{1,2,3}$,
					$\Delta_1^2=\set{4}$,
					$\Delta_2^2(1)=\set{5,6}$,
					$\Delta_2^2(2)=\set{7,8}$\\
		$(4,3,1)$ & $\Delta_1^1=\set{1,2,3,8}$,
					$\Delta_1^2=\set{4,6,7}$,
					$\Delta_2^2(1)=\set{5}$\\
	\end{tabular}
\end{table}

\begin{theorem}\label{Hamming_packing}
	If the extended Hamming code $\EH_k$
	is a packing $P_\pi$-code of packing radius $2$,
	then $\EH_k^{\varphi_\pi}$ is a packing $G_\pi$-code
	with packing radius $2$.
\end{theorem}
\begin{proof}
It follows from Theorem~\ref{covering_packing} that
$\EH_k^{\varphi_\pi}$ is a packing $G_\pi$-code.
By the partition condition in Proposition~\ref{main:iff},
there is $x$ in both
$S_{P_\pi}(c_1;3)$ and $S_{P_\pi}(c_2;3)$
for some distinct codewords $c_1$ and $c_2$ in $\EH_k$.
Then we have $w_{P_\pi}(c_1;x)\leq 3$ and $w_{P_\pi}(c_2;x)\leq 3$.
Applying Lemma~\ref{pro:sig-tau} $(iii)$, $(iv)$ to $w_{P_\pi}(c_i;x)$ we have
\[w_{G_\pi}(c_i^{\varphi_\pi}+x^{\varphi_\pi})=w_{G_\pi}((c_i+x)^{\varphi_\pi})=w_{P_\pi}(c_i+x)\leq 3,\]
where $i=1,2$.
This implies that $x^{\varphi_\pi}\in S_{G_\pi}(c_1^{\varphi_\pi};3)\cap S_{G_\pi}(c_2^{\varphi_\pi};3)$, and the result follows.
\end{proof}

\begin{example}
Let $\EH_3$ be a $2$-perfect $P_\pi$-code.
Consider weighted posets with the structure vector $(3,1,4)$.
Then the weighted posets in Fig.\ref{3,1,4} correspond to
the digraphs in Fig.\ref{3,1,4_graph}.
By Theorem~\ref{Hamming_packing},
the digraphs admit $\EH_3^{\varphi_\pi}$ to be a $2$-packing $G$-code. 
The codewords of $\EH_3^{\varphi_\pi}$
can be found in Table~\ref{3,1,4_tau}.

\begin{table}[h]
	\caption{}
	\small
	\label{3,1,4_tau}
	\begin{tabular}{>{\centering\arraybackslash}p{5cm} | >{\centering\arraybackslash}p{5cm}}
		\hline
		\multicolumn{2}{c}{$\Delta_1^1=\set{1,2,3}$,
					$\Delta_1^2=\set{4}$} \\
		\hline
			$\EH_3$ & $\EH_3^{\varphi_\pi}$ \\
		\hline
			1 2 3 4 5 6 7 8 & 1 2 3 4 $4'$5 6 7 8\\
			\hline
			0 0 0 0 0 0 0 0 & 0 0 0 0 0 0 0 0 0\\
			0 0 0 0 1 1 1 1 & 0 0 0 0 0 1 1 1 1\\
			1 0 0 1 0 1 1 0 & 1 0 0 1 0 0 1 1 0\\
			1 0 0 1 1 0 0 1 & 1 0 0 1 0 1 0 0 1\\
			0 1 0 1 1 0 1 0 & 0 1 0 1 0 1 0 1 0\\
			0 1 0 1 0 1 0 1 & 0 1 0 1 0 0 1 0 1\\
			1 1 0 0 1 1 0 0 & 1 1 0 0 0 1 1 0 0\\
			1 1 0 0 0 0 1 1 & 1 1 0 0 0 0 0 1 1\\
			0 0 1 1 1 1 0 0 & 0 0 1 1 0 1 1 0 0\\
			0 0 1 1 0 0 1 1 & 0 0 1 1 0 0 0 1 1\\
			1 0 1 0 1 0 1 0 & 1 0 1 0 0 1 0 1 0\\
			1 0 1 0 0 1 0 1 & 1 0 1 0 0 0 1 0 1\\
			0 1 1 0 0 1 1 0 & 0 1 1 0 0 0 1 1 0\\
			0 1 1 0 1 0 0 1 & 0 1 1 0 0 1 0 0 1\\
			1 1 1 1 0 0 0 0 & 1 1 1 1 0 0 0 0 0\\
			1 1 1 1 1 1 1 1 & 1 1 1 1 0 1 1 1 1\\
		\hline
	\end{tabular}
\end{table}
\begin{figure}[h]
	\begin{minipage}{2in}%
		\begin{tikzpicture}[auto,node distance=1cm,semithick]
			\tikzstyle{vertex}=[circle,fill,inner sep=1.5pt]

			\node[vertex] (a)			   {};
			\node[vertex] (b) [right of=a] {};
			\node[vertex] (c) [right of=b] {};
			\node[vertex] (d) [right of=c] {};
			
			\node[vertex] (e) [right of=d] {};
			\node[vertex] (e') [right of=e] {};		

			\node[vertex] (f) [below of=b] {};
			\node[vertex] (g) [below of=c] {};
			\node[vertex] (h) [below of=d] {};

			\path[->]
    		(a) edge (f)
    		(b) edge (f)
    		(c) edge (f)
    		(d) edge (f);

    		\path[->]
  			(e) edge [bend left] (e')
  			(e') edge [bend left] (e);

		\end{tikzpicture}
	\end{minipage}
\vskip1.5cm
	\begin{minipage}{2in}%
		\begin{tikzpicture}[auto,node distance=1cm,semithick]
			\tikzstyle{vertex}=[circle,fill,inner sep=1.5pt]
		
			\node[vertex] (a)			   {};
			\node[vertex] (b) [right of=a] {};
			\node[vertex] (c) [right of=b] {};
			\node[vertex] (d) [right of=c] {};
			
			\node[vertex] (e) [right of=d] {};
			\node[vertex] (e') [right of=e] {};	
		
			\node[vertex] (f) [below of=b] {};
			\node[vertex] (g) [below of=c] {};
			\node[vertex] (h) [below of=d] {};
		
			\path[->]
		    (a) edge (f)
		    (b) edge (f)
		    (c) edge (g)
		    (d) edge (g);

    		\path[->]
  			(e) edge [bend left] (e')
  			(e') edge [bend left] (e);

		\end{tikzpicture}
	\end{minipage}
	\caption{}
	\label{3,1,4_graph}
\end{figure}

\end{example}


\subsection{Graph Metrics} 

\label{sub:graph_metrics}
In this subsection, we consider a digraph $G$.
Recall that
$P_{\tilde{\pi}}$ denote the $\tilde{\pi}$-weighted poset induced by
a digraph.
It follows from~(\ref{P_G:sphere_size_radius_r}) that
\[\abs{S_{P_{\tilde{\pi}}}(0;2)}
	=1+\Omega_1^1(1)+\Omega_1^2(1)+2\Omega_1^2(2)+\Omega_2^2(2).\]
The number of vectors in $\FF^n$ whose $G$-distance to the zero vector
is exactly $d$ equals
\[
\left\{
\begin{array}{ll}
	1				& \text{if}~d=0,\\
	\Omega_1^1(1)	& \text{if}~d=1,\\
	3\Omega_1^2(1)+2\Omega_1^2(2)+\Omega_2^2(2) & \text{if}~d=2.
\end{array} \right.
\]
Therefore we have
\begin{equation}\label{G:sphere_size_radius_r}
	\abs{S_G(0;2)}
		=1+\Omega_1^1(1)+3\Omega_1^2(1)+2\Omega_1^2(2)+\Omega_2^2(2).
\end{equation}

\begin{lemma}\label{G:size}
	If the extended Hamming code $\EH_k$
	is a $2$-perfect $G$-code, then
	$\Omega_1^2(1)=1+\frac{1}{2}\Omega_1^1(1)(\Omega_1^1(1)-3)$.
\end{lemma}
\begin{proof}
Let $s$ be the number of elements in $P_G$ whose $P_G$-weight are $1$,
i.e. $s=\Omega_1^1(1)$.
It follows from the sphere packing condition and
(\ref{G:sphere_size_radius_r}) that
\begin{equation}\label{G:sphere_size}
	\abs{S_G(0;2)}=2^{k+1}
	=1+{s\choose1}+{s\choose2}+3\Omega_1^2(1)+2\Omega_1^2(2).
\end{equation}
Since $G$ is a digraph with $n$ vertices, we have $n=2^k$, and so
\begin{equation}\label{G:graph_size}
	\abs{V(G)}=2^k
	=s+2\Omega_1^2(1)+\Omega_1^2(2).
\end{equation}
The proof can be completed
from solving (\ref{G:sphere_size}) and~(\ref{G:graph_size}).
\end{proof}

As an illustration of our theorem,
we classify digraphs which admit the extended binary Hamming code
$\EH_3$ to be a $2$-perfect $G$-code.

\begin{lemma}\label{G:1-lebel}
	If the extended Hamming code $\EH_3$
	is a $2$-perfect $G$-code, then
	$1\leq\Omega_1^1(1)\leq3$.
\end{lemma}
\begin{proof}
It follows from~(\ref{G:graph_size}) that
	\begin{equation}\label{G:graph_size_k=3}
			\abs{V(G)}=2^3
				=\Omega_1^1(1)+2\Omega_1^2(1)+\Omega_1^2(2).
	\end{equation}
Note that $\Omega_1^2(2)$ is non-negative integer.
By Lemma~\ref{G:size}, the equality does not hold
if $\Omega_1^1(1)=0$ and $\Omega_1^1(1)\geq4$.
\end{proof}

It follows from Lemma~\ref{G:size}, Lemma~\ref{G:1-lebel}
and~(\ref{G:graph_size_k=3}) that the possible structure vectors are as follows:
\[(\Omega_1^1(1),\Omega_1^2(1),\Omega_1^2(2))
				=(1,0,7),(2,0,6),(3,1,3).\]
We classify the possible digraphs
which admit $\EH_3$ to be a $2$-perfect
$G$-code corresponding to the structure vectors.
The digraph $G$ has no cycles if $\Omega_1^1(1)\leq2$.
In this case, the digraph $G$ corresponds to a poset.
The authors classified the poset structures
which admit the extended binary Hamming code
to be a $2$-perfect poset code in~\cite{HK04}.
We now classify the digraphs for the remaining case $(3,1,3)$. There are three non-equivalent digraphs, say $G^1$, $G^2$ and $G^3$
with the structure vector $(3,1,3)$~(cf.~Fig.\ref{3,1,3} and~Fig.\ref{G:cf}).
\begin{figure}[h]
		\begin{tikzpicture}[auto,node distance=1cm,semithick]
			\tikzstyle{vertex}=[circle,fill,inner sep=1.5pt]
		
			\node[vertex] (a)				{};
			\node[vertex] (b) [right of=a]  {};
			\node[vertex] (c) [right of=b]  {};
			\node[vertex] (d) [right of=c]  {};
			\node[vertex] (h) [right of=d]  {};	
		
			\node[vertex] (e) [below of=a]  {};
			\node[vertex] (f) [below of=b]  {};
			\node[vertex] (g) [below of=c]  {};
		
			\path[->]
		    (a) edge  (e)
		    (b) edge  (f)
		    (c) edge  (g);
		
		    \path[->]
		  	(d) edge [bend left] (h)
		  	(h) edge [bend left] (d);
		
		\end{tikzpicture}
		\caption{}
		\label{3,1,3}
	\end{figure}
\begin{figure}[h]
\hskip1cm
	\begin{minipage}{1in}%
		\begin{tikzpicture}[auto,node distance=1cm,semithick]
			\tikzstyle{vertex}=[circle,fill,inner sep=1.5pt]
		
			\node[vertex] (a)				{};
			\node[vertex] (b) [right of=a]  {};
			\node[vertex] (c) [right of=b]  {};
			\node[vertex] (d) [right of=c]  {};
			\node[vertex] (h) [right of=d]  {};	
		
			\node[vertex] (e) [below of=a]  {};
			\node[vertex] (f) [below of=b]  {};
			\node[vertex] (g) [below of=c]  {};
		
			\path[->]
		    (a) edge  (e)
		    (b) edge  (e)
		    (c) edge  (e);
		
		    \path[->]
		  	(d) edge [bend left] (h)
		  	(h) edge [bend left] (d);
		
		\end{tikzpicture}
	\end{minipage}
\hskip3cm
	\begin{minipage}{2in}%
		\begin{tikzpicture}[auto,node distance=1cm,semithick]
			\tikzstyle{vertex}=[circle,fill,inner sep=1.5pt]
		
			\node[vertex] (a)				{};
			\node[vertex] (b) [right of=a]  {};
			\node[vertex] (c) [right of=b]  {};
			\node[vertex] (d) [right of=c]  {};
			\node[vertex] (h) [right of=d]  {};	
		
			\node[vertex] (e) [below of=a]  {};
			\node[vertex] (f) [below of=b]  {};
			\node[vertex] (g) [below of=c]  {};
		
			\path[->]
		    (a) edge  (e)
		    (b) edge  (e)
		    (c) edge  (g);
		
		    \path[->]
		  	(d) edge [bend left] (h)
		  	(h) edge [bend left] (d);
		
		\end{tikzpicture}
	\end{minipage}
	\caption{}
	\label{G:cf}
\end{figure}
\medskip

We define
\[\Gamma_i^\omega:=\set{v\in V(G)
					\mid\abs{\bar{v}}=i,w_G(v)=\omega}.\]

For any $v$ in $V(G)$, we define
\[\Gamma_i^\omega(v):=\set{u\in V(G)
							\mid (u,v)\in E(G)}\cap\Gamma_i^\omega.\]

\noindent
There are three codewords of $\EH_3$
which contains $\set{\alpha,\alpha'}$.
For each digraph,
we give a labeling such that $\Delta_2^2=\set{\alpha,\alpha'}$.
For $G^1$, we give a labeling as follows:
	$\Gamma_1^1=\set{\beta,\gamma,\delta}$,
	$\Gamma_1^2(\beta)=\set{\gamma'}$,
	$\Gamma_1^2(\gamma)=\set{\delta'}$ and
	$\Gamma_1^2(\delta)=\set{\beta'}$,
	where $\set{\alpha,\beta,\gamma,\delta}$,
			$\set{\alpha,\beta,\alpha',\beta'}$,
			$\set{\alpha,\gamma,\alpha',\gamma'}$ and
			$\set{\alpha,\delta,\alpha',\delta'}$ are codewords.
Then $G^1$ admits $\EH_3$ to be a $2$-perfect code.\\
For the remaining digraphs, we claim that they do not admit
$\EH_3$ to be a $2$-perfect code.
Let $\set{\alpha,\alpha',x,y}$ be a codeword.
The codewords which contains $\set{\alpha, \alpha'}$
are $\set{\alpha,\alpha',\beta,\beta'}$,
$\set{\alpha,\alpha',\gamma,\gamma'}$ or
$\set{\alpha,\alpha',\delta,\delta '}$.
Without loss of	generality,
$x\in\Gamma_1^1$ and $\abs{\Gamma_1^2(x)}\geq2$.
By Corollary~\ref{main:double}, we have
$y\in\Gamma_1^2\setminus\Gamma_1^2(x)$.
Therefore $G^2$ does not admit $\EH_3$ to be a
$2$-perfect code.
Since there is a unique vertex $x'$ such that $(y,x')$ is an edge,
we find an another codeword $\set{\alpha,\alpha',x',y'}$
which contains $\set{\alpha, \alpha'}$.
In a similar way, we have $y'\in\Gamma_1^2(x)$.
But $\set{x,x',y,y'}$ is a codeword,
which is a contradiction to Corollary~\ref{main:double}.\medskip

By combining the preceding discussion and the result ~\cite{HK04}, we obtain the following theorem.

\begin{theorem}
	Let $\EH_3$ denote the binary extended $[8,4,4]_H$ Hamming code.
	Then $\EH_3$ is a $2$-perfect $G$-code if and only if
	$G$ is one of the digraphs described 
	in Fig.\ref{3,1,3} or Fig.\ref{HK_graph}.
\end{theorem}

\begin{figure}[h]
	\begin{minipage}{1in}%
		\begin{tikzpicture}[auto,node distance=1cm,semithick]
			\tikzstyle{vertex}=[circle,fill,inner sep=1.5pt]
		
			\node[vertex] (1)							{};
			\node		  (0) [above of=1,yshift=5mm]   {$\cdots$};

			\node[vertex] (b) [right of=0]  {};
			\node[vertex] (c) [right of=b]  {};
			
			\node[vertex] (d) [left of=0]   {};
			\node[vertex] (e) [left of=d]   {};

			\path[<-]
			(1) edge (b)
			(1) edge (c)
			(1) edge (d)
			(1) edge (e);

			\draw[decorate,decoration={brace,mirror,raise=6pt,amplitude=10pt}, thick] 
			(c)--(e);
			\draw (0) node[above,yshift=6mm] {$7$};

		\end{tikzpicture}
	\end{minipage}
\hskip3cm
	\begin{minipage}{2in}%
		\begin{tikzpicture}[auto,node distance=1cm,semithick]
			\tikzstyle{vertex}=[circle,fill,inner sep=1.5pt]

			\node[vertex] (21)                {};
			\node[vertex] (11) [right of=21,xshift=5mm]  {};
			\node		  (01) [above of=11,yshift=5mm]   {$\cdots$};

			\node[vertex] (b1) [right of=01]  {};
			\node[vertex] (c1) [right of=b1]  {};
			
			\node[vertex] (d1) [left of=01]   {};
			\node[vertex] (e1) [left of=d1]   {};

			\path[<-]
			(11) edge (b1)
			(11) edge (c1)
			(11) edge (d1)
			(11) edge (e1);

			\draw[decorate,decoration={brace,mirror,raise=6pt,amplitude=10pt}, thick] 
			(c1)--(e1);
			\draw (01) node[above,yshift=6mm] {$6$};
		
		\end{tikzpicture}
	\end{minipage}
	\caption{Acyclic directed graphs which admit $\EH_3$ to be a 2-error correcting perfect code (cf.~\cite{HK04}).}
	\label{HK_graph}
\end{figure}
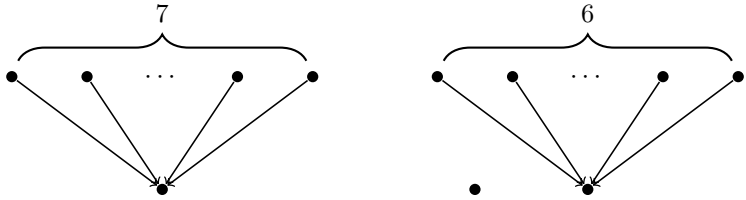

We provide a family of digraphs which admit
the extended binary Hamming code $\EH_k$ to be
a $2$-perfect $G$-code.

By Lemma~\ref{G:size} and~(\ref{G:graph_size}),
if $\EH_k$ is a $2$-perfect $G$-code, then we have the following possible structure vectors:
\[(\Omega_1^1(1),\Omega_1^2(1),\Omega_1^2(2))
			=(1,0,2^k-1),(2,0,2^k-2),(3,1,2^k-5),\ldots.\]

\begin{theorem}\label{G:infinite}
	Let $G$ be a digraph such that the following holds:
	\begin{enumerate}
	\item[] $(\Omega_1^1(1),\Omega_1^2(1),\Omega_1^2(2))
				=(3,1,2^k-5)$,\\
			$\Gamma_2^2=\set{\alpha,\alpha'}$,
			$\Gamma_1^1=\set{\beta,\gamma,\delta}$,
			$\Gamma_1^2(\beta)=\set{\gamma'}$,
			$\Gamma_1^2(\gamma)=\set{\delta'}$,\\
			$\Gamma_1^2(\delta)=
				V(G)\setminus(\Gamma_2^2\cup\Gamma_1^1
				\cup\set{\gamma',\delta'})$,
	\end{enumerate}
		where $\set{\alpha,\beta,\gamma,\delta}$,
		$\set{\alpha,\beta,\alpha',\beta'}$,
		$\set{\alpha,\gamma,\alpha',\gamma'}$ and
		$\set{\alpha,\delta,\alpha',\delta'}$ are codewords.
	Then the extended Hamming code $\EH_k$ is a
	$2$-perfect $G$-code.
\end{theorem}
\begin{proof}
We can check as in the preceding discussion that
Proposition~\ref{main:iff} and Corollary~\ref{main:double}
are satisfied.
\end{proof}

We provide concretely a possible labeling on $V(G)$
with Table~\ref{G:tab}
for which $\EH_3$ is a $2$-perfect $G$-code.

\begin{table}[h]
	\small
	\caption{}
	\label{G:tab}
	\begin{tabular}{p{2cm}|p{10cm}}
		\hline
		Structure & Possible labeling \\
		\hline
		$(3,1,4)$ & $\Gamma_1^1=\set{2,3,4}$,
					$\Gamma_2^2=\set{1,5}$,
					$\Gamma_1^2(2)=\set{8}$,
					$\Gamma_1^2(3)=\set{6}$,
					$\Gamma_1^2(4)=\set{7}$
	\end{tabular}
\end{table}

\begin{theorem}\label{Hamming_covering}
	If the extended Hamming code $\EH_k$
	is a $2$-perfect $G$-code, then
	$\EH_k^{\varphi_{\tilde{\pi}}}$ is a covering
	$P_{\tilde{\pi}}$-code of covering radius $2$.
\end{theorem}
\begin{proof}
Let $x$ be in $\FF^m$, where $m$ is the size of the $P_{\tilde{\pi}}$.
It is sufficient to show that $x$ is not in
$S_{P_{\tilde{\pi}}}(c^{\varphi_{\tilde{\pi}}};1)$ for any $c\in\EH_k$.
i.e. $w_{P_{\tilde{\pi}}}(c^{\varphi_{\tilde{\pi}}}+x)\geq2$.
It follows from Theorem~\ref{covering_packing} that
$\EH_k^{\varphi_{\tilde{\pi}}}$ is a covering $P_{\tilde{\pi}}$-code.
By Theorem~\ref{G:size}, if $\Omega_1^2(1)=0$,
then $\Omega_1^1(1)\leq2$.
In this case, the digraph $G$ corrsponds to just a poset.
This implies that $\EH_k^{\varphi_{\tilde{\pi}}}=\EH_k$
and the weighted poset $P_{\tilde{\pi}}$ induced by the digraph $G$  is a poset.
Then
$\EH_k^{\varphi_{\tilde{\pi}}}$ is $2$-perfect $P_{\tilde{\pi}}$-code,
and the result follows.
Thus we may assume that
there is $x$ in $\FF^m$ such that
$w_{P_{\tilde{\pi}}}(x)=2$ and $w_H(x)=1$.
Let $c$ be a non-zero codeword of $\EH_k$.
Note that $w_H(c)\geq4$
since $\EH_k$ has $4$ as the minimum Hamming distance.
Now, we claim that $w_H(c^{\varphi_{\tilde{\pi}}})\geq3$.
If $w_H(c^{\varphi_{\tilde{\pi}}})=1$,
then the digraph $G$ has a directed cycle of length at least $4$.
This is a contradiction to Lemma~\ref{main:weight2}.
If $w_H(c^{\varphi_{\tilde{\pi}}})=2$,
then the codeword $c$ can be partitioned into
$\set{a,b}$ such that
$w_H(a^{\varphi_{\tilde{\pi}}})=1$
and $w_H(b^{\varphi_{\tilde{\pi}}})=1$.
By Lemma~\ref{main:weight2}
we have
$w_G(a)=w_G(b)=2$ and $w_H(a)=w_H(b)=2$.
It follows that $w_G(c)=4$.
This is a contradiction to Corollary~\ref{main:double},
and our claim is proved.
Therefore we have
$2\leq w_H(c^{\varphi_{\tilde{\pi}}})-w_H(x)
\leq w_H(c^{\varphi_{\tilde{\pi}}}+x)
\leq w_{P_{\tilde{\pi}}}(c^{\varphi_{\tilde{\pi}}}+x)$,
and the proof is completed.
\end{proof}

\begin{example}
Let $\EH_3$ be a $2$-perfect $G$-code.
Then the digraph in Fig.~\ref{3,1,3} correspond to
the weighted poset in Fig.~\ref{3,1,3_poset}.
By Theorem~\ref{Hamming_covering},
the weighted poset admits $\EH_3^{\varphi_{\tilde{\pi}}}$ to be a $2$-covering $P_{\tilde{\pi}}$-code.
The codewords of $\EH_3^{\varphi_{\tilde{\pi}}}$ can be found
in Table~\ref{3,1,3_sigma}.

\begin{table}[h]
	\caption{}
	\small
	\label{3,1,3_sigma}
	\begin{tabular}{>{\centering\arraybackslash}p{6cm} | >{\centering\arraybackslash}p{6cm}}
		\hline
		\multicolumn{2}{c}{$\Gamma_1^1=\set{2,3,4}$,
					$\Gamma_2^2=\set{1,5}$,
					$\Gamma_1^2(2)=\set{8}$,
					$\Gamma_1^2(3)=\set{6}$,
					$\Gamma_1^2(4)=\set{7}$} \\
		\hline
			$\EH_3$ & $\EH_3^{\varphi_{\tilde{\pi}}}$ \\
		\hline
			1 2 3 4 5 6 7 8   &   2 3 4 5 6 7 8\\
			\hline
			0 0 0 0 0 0 0 0   &   0 0 0 0 0 0 0\\
			0 0 0 0 1 1 1 1   &   0 0 0 1 1 1 1\\
			1 0 0 1 0 1 1 0   &   0 0 1 1 1 1 0\\
			1 0 0 1 1 0 0 1   &   0 0 1 1 0 0 1\\
			0 1 0 1 1 0 1 0   &   1 0 1 1 0 1 0\\
			0 1 0 1 0 1 0 1   &   1 0 1 0 1 0 1\\
			1 1 0 0 1 1 0 0   &   1 0 0 1 1 0 0\\
			1 1 0 0 0 0 1 1   &   1 0 0 1 0 1 1\\
			0 0 1 1 1 1 0 0   &   0 1 1 1 1 0 0\\
			0 0 1 1 0 0 1 1   &   0 1 1 0 0 1 1\\
			1 0 1 0 1 0 1 0   &   0 1 0 1 0 1 0\\
			1 0 1 0 0 1 0 1   &   0 1 0 1 1 0 1\\
			0 1 1 0 0 1 1 0   &   1 1 0 0 1 1 0\\
			0 1 1 0 1 0 0 1   &   1 1 0 1 0 0 1\\
			1 1 1 1 0 0 0 0   &   1 1 1 1 0 0 0\\
			1 1 1 1 1 1 1 1   &   1 1 1 1 1 1 1\\
		\hline
	\end{tabular}
\end{table}
\begin{figure}[h]
	\begin{tikzpicture}[auto,node distance=1cm,semithick]
		\tikzstyle{weight}=[shape=circle,draw,inner sep=0.1pt]
	
		\node[weight] (a)				  {$1$};
		\node[weight] (b) [right of=a]  {$1$};
		\node[weight] (c) [right of=b]  {$1$};
		
		\node[weight] (d) [right of=c]  {$2$};

		\node[weight] (e) [below of=a]  {$1$};
		\node[weight] (f) [below of=b]  {$1$};
		\node[weight] (g) [below of=c]  {$1$};
	
		\path[-]
	    (a) edge  (e)
	    (b) edge  (f)
	    (c) edge  (g);
	
	\end{tikzpicture}
	\caption{}
	\label{3,1,3_poset}
\end{figure}

\end{example}



\end{document}